\documentclass[openright, 12pt]{article}
\usepackage{amsmath,amsfonts}
\usepackage{amsthm}
\newtheorem{thm}{Theorem}

\newtheorem{lem}[thm]{Lemma}

\newtheorem{qu}[thm]{Question}
\usepackage{float} 
\usepackage{multirow} 
\usepackage[dvips]{graphicx}
\usepackage{psfrag}
\usepackage{enumerate}

\usepackage{amssymb}
\usepackage[latin1]{inputenc}

\begin{document}

\begin{center}

\noindent  \large Fundamental groups of aspherical manifolds
 
  that collapse.\normalsize

\begin{normalsize}
Sergio Zamora

sxz38@psu.edu
\end{normalsize}
\end{center}

\begin{abstract}
We show that if a sequence $M_n$ of closed aspherical $d$-dimensional Riemannian manifolds with Ricci curvature uniformly bounded below and diameter uniformly bounded above collapses, then for all large enough $n$, the fundamental groups $\pi_1(M_n)$ have non-trivial finitely generated abelian normal subgroups. In particular, the groups $\pi_1(M_n)$ cannot be non-elementary hyperbolic.
\end{abstract}

\section{Introduction}

In the 80's Gromov \cite{PG} showed that for fixed $c \in \mathbb{R}$, $d \in \mathbb{N}$, and $D > 0$, the family $ \mathfrak{M}(d,c,D)$ of closed Riemannian manifolds $M$  of dimension  $ dim(M) \leq d$,  Ricci curvature $Ric(M) \geq c$,  and diameter $diam(M ) :  = \sup_{x,y \in M }d(x,y) \leq D $ is pre-compact in the Gromov--Hausdorff topology. 

A lot can be said about both the elements of $\mathfrak{M}(d,c,D)$ and their Gromov--Hausdorff limit points (see \cite{CC}, \cite{Co}, \cite{PG}, \cite{LV}, \cite{Wei}). A sequence $M_n \in \mathfrak{M}(d,c,D)$ behaves quite differently depending on whether $vol(M_n) \geq \nu > 0$ for all $n$, or $vol(M_n)\to 0$ as $n \to \infty$ (up to subsequence, one of these two events occur) (see \cite{CJN}, \cite{GPW}, \cite{KL}, \cite{Wa}). In the latter case we say that the sequence \textit{collapses}. The goal of this note is to compile a proof of the following results:

\begin{thm}\label{Abelian}
Let $M_n \in \mathfrak{M}(d,c,D)$ be a collapsing sequence of aspherical manifolds. Then for large enough $n$, there are non-trivial finitely generated abelian normal subgroups $1 \neq H_n \triangleleft \pi_1 (M_n)$.
\end{thm}

\begin{thm}\label{HyperCollapse}
Let $M_n \in \mathfrak{M}(d,c,D)$ be a sequence of aspherical manifolds with $\pi_1(M_n) $ non-elementary hyperbolic for each $n$. Then such sequence is non-collapsing.
\end{thm}

It is a classical result that aspherical closed manifolds with hyperbolic fundamental group have positive simplicial volume (see \cite{GH}, \cite{LohS}), so Theorem \ref{HyperCollapse} is known to hold when all $M_n's$ are homeomorphic to each other.

The structure of this note is as follows: Section \ref{Prelim} covers the ingredients for Theorems \ref{Abelian} and \ref{HyperCollapse}, Section \ref{P} contains the proof of such Theorems, and Section \ref{Problems} discusses related questions and open problems.

\section{Preliminaries}\label{Prelim}

\subsection{Aspherical Manifolds}

Let $M$ be a closed smooth manifold. We say that $M$ is \textit{aspherical} if $\pi_k(M) = 0$ for $k \geq 2$. 

\begin{thm}\label{Sys}
(\cite{FRM}, Section 1). Let $M_n$ be a sequence of closed aspherical $d$-dimensional Riemannian manifolds such that $vol(M_n) \to 0$ as $n \to \infty$. Then there is a sequence of noncontractible loops $\gamma_n : \mathbb{S}^1 \to M_n$ such that $length (\gamma_n) \to 0$ as $n \to \infty$.
\end{thm}

\begin{lem}\label{TorFree}
(\cite{H}, Proposition 2.45). Let $M$ be a closed aspherical $d$-dimensional manifold, then $\pi_1(M)$ is torsion free.
\end{lem}

\subsection{Ricci Curvature Bounds}

The most important ingredient for Theorem \ref{Abelian} is the following result by Vitali Kapovitch and Burkhard Wilking.

\begin{thm}\label{KW}
(\cite{KW}, Theorem 6) For $d \in \mathbb{N},$ $c \in \mathbb{R}$, and $D > 0$,  there are positive constants $\varepsilon_0$ and $ C $ such that for each $M \in \mathfrak{M}(d,c,D)$ there are $\varepsilon \geq  \varepsilon _0$ and a normal nilpotent subgroup $N \triangleleft \pi_1 (M)$ of rank and step $\leq d$  satisfying that for each $p \in \mathbb{M}$, the image of the map 
$$ \pi_1 ( B (p, \varepsilon),p) \to \pi_1 (M,p) $$
given by the inclusion $B(p, \varepsilon ) \to M$ contains $N$ as a subgroup of index $\leq C$.
\end{thm}

\subsection{Hyperbolic Groups}\label{HyperbolicAlgebra}



A hyperbolic group $\Gamma$ is called \textit{elementary} if it contains a cyclic subgroup of finite index.

\begin{thm}\label{Z2}
(\cite{Loh}, Corollary 7.5.19). Let $\Gamma$ be a non-elementary hyperbolic group. Then no subgroup of $\Gamma$ is isomorphic to $\mathbb{Z}^2$.
\end{thm}

\begin{thm}\label{Normalizer}
(\cite{GdlH}, Theorem 34), (\cite{O}, Lemma 1.16). Let $h$ be an infinite order element of a hyperbolic group $\Gamma$, and $C : = \langle h \rangle$. Then the set
$$ E :=  \{ g \in \Gamma  \vert (g C g^{-1} ) \cap C \neq \{ e \}   \}  $$
is a subgroup of $\Gamma$ containing $C$, and $ [ E : C ] < \infty $.
\end{thm}

\section{Proofs of Theorems \ref{Abelian} and \ref{HyperCollapse}}\label{P}

\begin{proof}[Proof of Theorem \ref{Abelian}]
By Theorem \ref{KW}, there is a sequence $\varepsilon_n \geq \varepsilon_0 >0$ and normal nilpotent subgroups $N_n \triangleleft \pi_1 (M_n)$ of rank and step $\leq d$ with the property that for each sequence $p_n \in M_n$, the images of the maps 
$$\pi_1 ( B(p_n , \varepsilon_n ), p_n  )  \to \pi_1 (M_n, p_n )      $$ 
given by the inclusions $B(p_n, \varepsilon_n) \to M_n$ contain $N_n$ as finite index subgroups.
By Theorem \ref{Sys}, there is a sequence of points $x_n \in M_n$ and $r_n \searrow 0$ such that the maps 
$$\pi_1 (  B(x_n,r_n ) , x_n   ) \to \pi_1 (M_n , x_n)    $$
given by the inclusions $B(x_n, r_n) \to M_n$ have non-trivial image.
For large enough $n$, we have $r_n \leq \varepsilon_0$, and the maps 
$$\pi_1 (  B(x_n,\varepsilon_n ) , x_n   ) \to \pi_1 (M_n , x_n)    $$
have non-trivial image. By Lemma \ref{TorFree}, such images are infinite, implying that their finite index subgroups $N_n$ are non-trivial.

Since the groups $N_n$ are finitely generated nilpotent, their centers $H_n \leq N_n$ are non-trivial finitely generated abelian, and they are preserved by any automorphism of $N_n$.  This means that the adjoint action of $\pi_1(M_n, x_n)$ on $N_n$ preserves $H_n$, proving the theorem.
\end{proof}

\begin{proof}[Proof of Theorem \ref{HyperCollapse}:]
 By contradiction, assume there is a collapsing sequence $M_n \in \mathfrak{M}(d,c,D)$ consisting of aspherical manifolds with non-elementary hyperbolic fundamental groups $\pi_1(M_n)$.

 By Theorem \ref{Abelian}, for all large $n$, there are non-trivial finitely generated (torsion free, by Lemma \ref{TorFree}) abelian normal subgroups $H_n \triangleleft \pi_1(M_n)$. By Theorem \ref{Z2}, the groups $H_n$ are cyclic, and by Theorem \ref{Normalizer}, $[\pi_1(M_n) : H_n ] < \infty$, contradicting the fact that the groups $\pi_1(M_n)$ are non-elementary. 
\end{proof}

\section{Further Problems}\label{Problems}

The most natural question to follow up Theorems \ref{Abelian} and \ref{HyperCollapse} is whether the diameter hypothesis could be removed.

\begin{qu}
Let $M_n$ be a sequence of aspherical $d$-dimensional closed smooth Riemannian manifolds with non-elementary hyperbolic fundamental groups and $Ric(M_n) \geq -1$. Is it possible that $vol(M_n)\to 0$ as $n \to \infty$?
\end{qu}

\begin{qu}
Let $M$ be a $d$-dimensional closed smooth manifold admitting a metric of constant sectional curvature $-1$, and $D > 0$. How small 
\[    \spadesuit (M)   :=  \inf \{ diam (M,g)  \mid g \text{ Riemannian, } Ric(M,g) \geq  -1    \}      \]
\[   \clubsuit (M)   :=  \inf \{ vol (M,g)  \mid g \text{ Riemannian, } Ric(M,g) \geq  -1    \}       \]
\[      \clubsuit (M,D)   :=  \inf \{ vol (M,g)  \mid g \text{ Riemannian, } Ric(M,g) \geq  -1 , diam(M,g) \leq D   \}       \]
can be? Could $\spadesuit (M)$, $ \clubsuit (M) $, $ \clubsuit (M,D) $ be estimated in terms of the complexity of $\pi_1 (M)$?
\end{qu}

\end{document}